\documentclass[a4paper,11pt]{article}

\usepackage{amsmath,amsthm}
\usepackage{amssymb}
\usepackage{breqn}
\usepackage{enumerate}
\usepackage{graphicx}
\usepackage{bm}
\usepackage{bbm}
\usepackage[affil-it]{authblk}
\usepackage{tabu}
\usepackage{bold-extra}
\usepackage[hang,flushmargin]{footmisc}
\usepackage[hypertex]{hyperref}
\usepackage{mathtools}
\usepackage{relsize}
\usepackage{scalerel}
\usepackage{xcolor}
\usepackage{titlesec}
\usepackage{apptools}
\usepackage{appendix}

\newtheorem{theorem}{Theorem}
\newtheorem{corollary}[theorem]{Corollary}
\newtheorem{lemma}[theorem]{Lemma}
\newtheorem{proposition}[theorem]{Proposition}

\newtheorem{claim}[theorem]{Claim}

\theoremstyle{definition}
\newtheorem{defin}[theorem]{Definition}

\AtAppendix{\counterwithin{theorem}{section}}

\titleformat{\section}[hang]{\scshape\large\bfseries\filcenter}{\S\thesection}{4pt}{}
\titleformat{\subsection}[hang]{\scshape\bfseries}{\thesubsection.}{4pt}{}

\allowdisplaybreaks

\newcommand{\dhor}{\mathsf{D}_{\mathsf{hor}}}
\newcommand{\dver}{\mathsf{D}_{\mathsf{ver}}}

\newcommand{\on}[1]{
	\operatorname{#1}
}

\newcommand{\tightoverset}[2]{
  \mathop{#2}\limits^{\vbox to -.5ex{\kern-1.15ex\hbox{$#1$}\vss}}}

\newcommand\restr[2]{{
  \left.\kern-\nulldelimiterspace 
  #1 
  \vphantom{\big|} 
  \right|_{#2} 
}}

\makeatletter
\newcommand{\subalign}[1]{%
  \vcenter{%
    \Let@ \restore@math@cr \default@tag
    \baselineskip\fontdimen10 \scriptfont\tw@
    \advance\baselineskip\fontdimen12 \scriptfont\tw@
    \lineskip\thr@@\fontdimen8 \scriptfont\thr@@
    \lineskiplimit\lineskip
    \ialign{\hfil$\m@th\scriptstyle##$&$\m@th\scriptstyle{}##$\hfil\crcr
      #1\crcr
    }%
  }%
}
\makeatother

\newcommand\blfootnote[1]{%
  \begingroup
  \renewcommand\thefootnote{}\footnote{#1}%
  \addtocounter{footnote}{-1}%
  \endgroup
}

\newcommand{\exx}{
  \mathop{
    \mathchoice{\vcenter{\hbox{\larger[4]$\mathbb{E}$}}}
               {\kern0pt\mathbb{E}}
               {\kern0pt\mathbb{E}}
               {\kern0pt\mathbb{E}}
  }\displaylimits
}

\makeatletter
\newcommand*\bcdot{\mathpalette\bigcdot@{0.5}}
\newcommand*\bigcdot@[2]{\mathbin{\vcenter{\hbox{\scalebox{#2}{$\m@th#1\bullet$}}}}}
\makeatother

\makeatletter
\def\blfootnote{\gdef\@thefnmark{}\@footnotetext}
\makeatother

\newcommand\id{\mathbbm{1}}

\begin{document}
\begin{center}\Large\noindent{\bfseries{\scshape A note on transverse sets and bilinear varieties}}\\[24pt]\normalsize\noindent{\scshape Luka Mili\'cevi\'c\dag}\\[6pt]
\end{center}
\blfootnote{\noindent\dag\ Mathematical Institute of the Serbian Academy of Sciences and Arts\\\phantom{\dag\ }Email: luka.milicevic@turing.mi.sanu.ac.rs}

\footnotesize
\begin{changemargin}{1in}{1in}
\centerline{\sc{\textbf{Abstract}}}
\phantom{a}\hspace{12pt}~Let $G$ and $H$ be finite-dimensional vector spaces over $\mathbb{F}_p$. A subset $A \subseteq G \times H$ is said to be \emph{transverse} if all of its rows $\{x \in G \colon (x,y) \in A\}$, $y \in H$, are subspaces of $G$ and all of its columns $\{y \in H \colon (x,y) \in A\}$, $x \in G$, are subspaces of $H$. As a corollary of a bilinear version of Bogolyubov argument, Gowers and the author proved that dense transverse sets contain bilinear varieties of bounded codimension. In this paper, we provide a direct combinatorial proof of this fact. In particular, we improve the bounds and evade the use of Fourier analysis and Freiman's theorem and its variants. 
\end{changemargin}
\normalsize
\section{Introduction}

In this paper, we fix a prime field $\mathbb{F}_p$ and consider finite-dimensional vector spaces over $\mathbb{F}_p$, usually denoted by $G$ or $H$. A classical result, due to Bogolyubov, adapted to the finite vector space setting, reads as follows.

\begin{theorem}[Bogolyubov argument~\cite{BogolyubovPaper}]Suppose that $G$ is a finite-dimensional vector space over $\mathbb{F}_p$. Let $A \subseteq G$ be a set of size at least $\delta |G|$. Then $2A - 2A = \{a_1 + a_2 - a_3 - a_4 \colon a_1, a_2, a_3, a_4 \in A\}$ contains a subspace of $G$ of codimension $O_\delta(1)$.\end{theorem}

Motivated by applications in number theory~\cite{Mobius} and higher order Fourier analysis~\cite{U4paper}, Bienvenu and L\^e~\cite{BienvenuLe}, and Gowers and the author~\cite{BilinearBog}, respectively, independently proved a bilinear generalization of Bogolyubov argument. It states that if we are given a dense set $A$ inside $G \times H$, then replacing rows or columns by their difference sets several times, we may find bilinear structure in the resulting set. To be precise, let us write $\dhor A$ for the \emph{horizontal difference set} of $A$ which is defined as $\dhor A = \{(x_1 - x_2, y) \colon (x_1, y), (x_2, y) \in A\}$. Similarly, we write $\dver A$ for the \emph{vertical difference set} of $A$ which is defined as $\dver A = \{(x, y_1 - y_2) \colon (x, y_1), (x, y_2) \in A\}$.

\begin{theorem}[Bilinear Bogolyubov argument~\cite{BienvenuLe,BilinearBog}]\label{bilbogintro}Let $A \subseteq G \times H$ be a set of size at least $\delta |G||H|$. Then
\[\dver\dver\dhor\dhor\dver A\]
contains a set of the form $\{(x,y) \in U \times V \colon \beta(x,y) = 0\}$, where $U \leq G$ and $V \leq H$ are subspaces of codimensions $O_\delta(1)$ and $\beta \colon G \times H \to \mathbb{F}_p^r$ is a bilinear map, where $r = O_\delta(1)$.\end{theorem}

We refer to the sets of the form $\{(x,y) \in U \times V \colon \beta(x,y) = 0\}$, like the ones appearing in the conclusion above, as \emph{bilinear varieties}. We also refer to the quantity $\on{codim}_G U + \on{codim}_H V + r$ as the \emph{codimension} of the bilinear variety.\\
\indent The bounds in Theorem~\ref{bilbogintro} were later improved by Hosseini and Lovett~\cite{HosseiniLovett} to nearly optimal ones, namely $\on{codim}_G U, \on{codim}_H V, r \leq O(\log^{80}(2\delta^{-1}))$, at the small cost of considering a slightly longer sequence of directional difference maps.\\

\indent In~\cite{BilinearBog}, it was observed that if a dense set $A$ is invariant under maps $\dver$ and $\dhor$ then Theorem~\ref{bilbogintro} implies that $A$ contains a large bilinear variety of small codimension. Such sets were named \emph{transverse} by Bienvenu, Gonz\'alez-S\'anchez and Mart\'inez~\cite{TransverseSets}.

\begin{defin}A subset $A \subseteq G \times H$ is said to be \emph{transverse} if $\{y \in H \colon (x,y) \in A\}$ is a non-empty\footnote{The restriction that the subspaces are non-empty is a minor one, as if we allowed empty subspaces, we could have passed to product $\{x \in G \colon (x,0) \in A\} \times \{y \in H \colon (0,y) \in A\}$.} subspace of $H$ for each $x \in G$ and $\{x \in G \colon (x,y) \in A\}$ is a non-empty subspace of $G$ for each $y \in H$.\end{defin}

In this language, Theorem~\ref{bilbogintro}, with bounds due to Hosseini and Lovett, implies the following corollary.

\begin{corollary}\label{transverseCor}Let $A \subseteq G \times H$ be a transverse set of size at least $\delta |G||H|$. Then $A$ contains a bilinear variety of codimension $O(\log^{80}(2\delta^{-1}))$.\end{corollary}

\indent In~\cite{TransverseSets}, Bienvenu, Gonz\'alez-S\'anchez and Mart\'inez considered the problem of understanding the relationship between transverse sets and bilinear varieties. Interestingly, they constructed examples of transitive sets that are not bilinear varieties, indicating that some subtleties are likely to arise in the proof of Corollary~\ref{transverseCor}.\\
\indent Despite the significant algebraic structure present in the assumptions, the only known proof of Corollary~\ref{transverseCor} proceeds via bilinear Bogolyubov argument, which is capable of generating algebraic structure starting from sets that are merely dense. Here we provide a different, more direct proof of Theorem~\ref{transverseCor}, giving improved bounds. Our proof also identifies the relevant bilinear map combinatorially, rather than by using Fourier analysis or Freiman's theorem and related results about approximate homomorphisms. The main result of this paper is the following.

\begin{theorem}\label{bilinearSetsStructure}Let $A \subseteq G \times H$ be a transverse set of size at least $\delta |G||H|$. Then $A$ contains a set of the form $\{(x,y) \in U \times V \colon \beta(x,y) = 0\}$, where $U \leq G$ and $V \leq H$ are subspaces and $\beta \colon G \times H \to \mathbb{F}_p^r$ is a bilinear map, obeying the bounds $\on{codim}_G U = O(\log_p^3 \delta^{-1}), \on{codim}_H V = O(\log_p^2 \delta^{-1})$ and $r = O(\log_p \delta^{-1})$.\end{theorem}

Let us remark the bound $r = O(\log_p \delta^{-1})$ on the codimension of the bilinear map $\beta$ above is optimal up to constant. We also remark that the proof by Hosseini and Lovett implies that $\on{codim}_G U = O(\log^{80}(2\delta^{-1})), \on{codim}_H V = O(\log_p^{28}(2\delta^{-1}))$ and $r = O(\log^4 (2\delta^{-1}))$.\\

Finally, a multidimensional generalization of Corollary~\ref{transverseCor} was obtained in~\cite{Multihom}, again depending on results in the spirit of bilinear Bogolyubov argument (appropriately generalized). We expect that arguments of this paper could be useful in finding a direct proof of the multidimensional generalization of Corollary~\ref{transverseCor}, but, as we shall remark later in the paper, this approach would no longer be elementary in the higher order setting.\\

\noindent\textbf{Acknowledgements.} This work was supported by the Ministry of Science, Technological Development and Innovation of the Republic of Serbia through the Mathematical Institute of the Serbian Academy of Sciences and Arts.

\section{Proof of Theorem~\ref{bilinearSetsStructure}}

\noindent\textbf{Notation.} In this paper, we shall frequently consider subsets of products of two vectors spaces $G$ and $H$. Given $X \subset G \times H$ and an element $x \in G$, we write $X_{x \bcdot} = \{y \in H \colon (x,y) \in X\}$ for the \emph{(vertical) slice} of $X$ in the column indexed by $x$. Likewise, for an element $y \in H$, we write $X_{\bcdot y} = \{x \in G \colon (x,y) \in X\}$ for the \emph{(horizontal) slice} of $X$ in the row indexed by $y$.\\

In order to study a transverse set $A$, we move to the set of orthogonal complements of its columns. Namely, for each $x \in G$, we consider the subspace $V_x = A_{x \bcdot}^\perp$, where ${}^\perp$ is taken with respect to some fixed inner product on $G$. Observe that these subspaces satisfy $V_{x_1 + x_2} \subseteq V_{x_1} + V_{x_2}$. Indeed, since $A$ is a transverse set, its rows are subspaces so $A_{x_1 \bcdot} \cap A_{x_2 \bcdot} \subseteq A_{x_1 + x_2 \bcdot}$. Therefore
\[V_{x_1 + x_2} = A_{x_1 + x_2 \bcdot}^\perp \subseteq (A_{x_1 \bcdot}  \cap A_{x_2 \bcdot})^\perp = A_{x_1 \bcdot}^\perp + A_{x_2 \bcdot}^\perp =  V_{x_1} + V_{x_2},\]
as claimed. This motivates the following definition.

\begin{defin}Let $V_x$ be a non-empty subspace of $H$ for each $x \in G$ and assume that $V_0 = \{0\}$. We say that $(V_x)_{x \in G}$ is a \emph{linear system of subspaces inside} $H$ if $V_{x_1 + x_2} \subseteq V_{x_1} + V_{x_2}$ holds for all $x_1, x_2 \in G$.\end{defin}

Observe that in a transverse set $A$ we have $A_{0 \bcdot} = H$, so $A_{0 \bcdot}^\perp = \{0\}$. Thus, every transverse set gives rise to a linear system of subspaces in a natural way. Likewise, it turns out that every linear system of subspaces produces a transverse set.

\begin{lemma}Let $(V_x)_{x \in G}$ be a linear system of subspaces inside $H$. Let $A = \cup_{x \in G} \{x\}\times V_x^\perp$. Then $A$ is a transverse set.\end{lemma}

\begin{proof}Note that all columns of $A$ are non-empty as each $V_x$ is non-empty, and that all rows are non-empty since $V_0 = \{0\}$. Let $x_1, x_2, y \in G$ be elements such that $(x_1, y), (x_2, y) \in A$. Then $y \in V_{x_1}^\perp \cap V_{x_2}^\perp = (V_{x_1}+ V_{x_2})^\perp $. Since $(V_x)_{x \in G}$ is a linear system of subspaces, we also have that $y \in V_{x_1 + x_2}^\perp$, so $(x_1 + x_2, y) \in A$, proving that the rows of $A$ are subspaces. On the other hand, the columns of $A$ are subspaces by definition.\end{proof}

We remark also the following properties of a linear system of subspaces.

\begin{lemma}\label{basicpropsLSS}Let $(V_x)_{x \in G}$ be a linear system of subspaces inside $H$.\begin{itemize}
\item[\textbf{(i)}] For all $\lambda \in \mathbb{F}_p \setminus \{0\}$ and $x \in G$ we have $V_{\lambda x} = V_x$.
\item[\textbf{(ii)}] Whenever $x_1, \dots, x_r \in G$ add up to 0, we have $V_{x_1} + \dots + V_{x_r} = V_{x_1} + \dots + V_{x_{r-1}}$.
\end{itemize}\end{lemma}

\begin{proof}
\textbf{Proof of \textbf{(i)}.} Since $(V_x)_{x \in G}$ is a linear system of subspaces, we have
\[V_{\lambda x} \subseteq V_{(\lambda - 1) x} + V_x \subseteq (V_{(\lambda - 2) x} + V_x) + V_x \subseteq \dots \subseteq \lambda V_x = V_x.\]
proving the first claim.\\
\textbf{Proof of \textbf{(ii)}.} Suppose that $x_1 + \dots + x_r = 0$. Then, by the first part, $V_{x_r} = V_{-x_1-\dots -x_{r-1}} =  V_{x_1+\dots +x_{r-1}}$. Since $(V_x)_{x \in G}$ is a linear system of subspaces, we have
\[V_{x_1+\dots +x_{r-1}} \subseteq V_{x_1+\dots +x_{r-2}} + V_{x_{r-1}} \subseteq \dots\subseteq V_{x_1} + \dots V_{x_{r-1}},\]
so $V_{x_r} \subseteq V_{x_1} + \dots V_{x_{r-1}}$, from which the second claim follows.\end{proof}

The following lemma allows us to pass from subspaces to linear maps. It is an exact version of Lemma 34 from~\cite{AQV}, obtained by taking $K = 1$ in that paper.

\begin{lemma}\label{addquadsubspaces}Let $U_1, U_2, U_3, U_4 \leq H$ be subspaces of dimension $d$ such that $|U_{i_1} + U_{i_2} + U_{i_3}| = p^{3d}$ for any three distinct indices $i_1, i_2, i_3 \in [4]$ and such that $|U_1 + U_2 + U_3 + U_4| = p^{3d}$. Let $\phi_4 \colon \mathbb{F}_p^d \to U_4$ be a linear isomorphism. Then there exist linear isomorphisms $\phi_i \colon \mathbb{F}_p^d \to U_i$, $i \in [3]$, such that $\phi_1 + \phi_2 = \phi_3 + \phi_4$.\end{lemma}

It is important that we get further maps $\phi_1, \phi_2$ and $\phi_3$ for any given $\phi_4$, rather than only being able to find a single quadruple of maps $(\phi_1, \dots, \phi_4)$ satisfying the properties in the conclusion of the lemma.\\
\indent Furthermore, we need the related uniqueness result, which is an exact version of Lemma 35 from~\cite{AQV}, obtained by taking $K = 1$ and $r = 0$ in that paper.

\begin{lemma}\label{indepsubspace}Let $U_1, U_2, V_1, V_2, W \leq H$ be subspaces of dimension $d$ such that $W \cap (U_1 + U_2 + V_1 + V_2) = \{0\}$. Suppose that linear maps $\phi_i \colon \mathbb{F}_p^d \to U_i$ for $i \in [2]$, $\psi_i \colon \mathbb{F}_p^d \to V_i$ for $i \in [2]$ and $\theta \colon \mathbb{F}_p^d \to W$ satisfy $\phi_1 + \phi_2 + \psi_1 + \psi_2 + \theta = 0$. Then $\theta = 0$.\end{lemma}

We shall also make use of elementary fact that near-homomorphisms of vector spaces come from (exact) homomorphisms. The formulation below is a slight rephrasing of Corollary 26 in~\cite{Multihom}.

\begin{lemma}\label{elemnearhomm}Let $G_1$ and $G_2$ be two finite-dimensional vector spaces over $\mathbb{F}_p$. Let $A \subseteq G_1$ be a subset and let $\phi \colon A \to G_2$ be a map such that $x + y - z\in A$ and $\phi(x) + \phi(y) = \phi(z) + \phi(x + y - z)$ both hold for at least $(1 - \varepsilon)|G_1|^3$ triples $(x,y, z) \in A^3$. Provided $\varepsilon < 10^{-5}$, there exists an affine map $\Phi \colon G_1 \to G_2$ such that $\Phi(x) = \phi(x)$ holds for at least $(1 - 5\sqrt[4]{\varepsilon})|G_1|$ of elements $x\in G_1$.\end{lemma}

\noindent\textbf{Remark.} Returning to the last remark from the introduction, in order to prove the multidimensional generalization of Corollary~\ref{transverseCor} using the approach in this paper, we would need a higher-order generalization of Lemma~\ref{elemnearhomm}. Such a generalization concerns extending multilinear maps defined on $1-o(1)$ proportion of multilinear varieties. Unlike the lemma above, such a higher-order result is highly non-trivial, and, in fact, is one of the main results from~\cite{Multihom}.\\

The heart of the proof of Theorem~\ref{bilinearSetsStructure} is that a quasirandom linear system of subspaces is generated by a bilinear map of expected codimension, which is the content of the next proposition.

\begin{proposition}\label{propfindingbilstruct}Let $(V_x)_{x \in G}$ be a linear system of subspaces inside $H$. Let $d$ be a non-negative integer and let $\varepsilon > 0$ be such that
\begin{itemize}
\item[\textbf{(i)}] for at least $(1 - \varepsilon)|G|$ of elements $x \in G$ we have $|V_x| = p^d$, and
\item[\textbf{(ii)}] for at least $(1 - \varepsilon)|G|^2$ of pairs $(x_1, x_2) \in G^2$ we have $V_{x_1} \cap V_{x_2} = \{0\}$.
\end{itemize}
Provided $\varepsilon \leq p^{-8d} 10^{-32}$, there exist a bilinear map $\Phi \colon G \times \mathbb{F}_p^d \to H$ and a linear map $\Psi \colon \mathbb{F}_p^d \to H$ such that $\{\Phi(x, \lambda) + \Psi(\lambda) \colon \lambda \in \mathbb{F}_p^d\} = V_x$ holds for at least $(1 - 200p^d \sqrt[16]{\varepsilon})|G|$ elements $x \in G$.
\end{proposition}

We remark that Proposition~\ref{propfindingbilstruct} is similar in spirit to Theorem 33 from~\cite{AQV}, and these two results can roughly be viewed as `99\%' and `1\%' cases of the same phenomenon. Given the additional structure in the setting of the present paper, we are able to obtain a more streamlined proof. An important difference from~\cite{AQV} is that throughout that paper one has sufficiently strong quasirandom properties, which is not the case in this paper. For that reason, we need an efficient regularity lemma for transverse sets, which we prove later.

\begin{proof}[Proof of Proposition~\ref{propfindingbilstruct}] Let us first observe that our assumptions imply that $V_{x_0} \cap (V_{x_1} + \dots + V_{x_r}) = \{0\}$ holds for all but at most $\Big(p^{rd} \sqrt{\varepsilon} + r\varepsilon \Big)|G|^{r+1}$ of $(r+1)$-tuples $(x_0, \dots, x_r) \in G^{[0,r]}$. To see this, assume the contrary. Let $\mathcal{G}$ be the set of all $x \in G$ such that $|V_x| = p^d$. The total number of $(r+1)$-tuples $(x_0, \dots, x_r) \in G^{[0,r]}$ such that one of $x_1, \dots, x_r$ is not in $\mathcal{G}$ is at most $r \varepsilon |G|^{r+1}$. Hence we get more than $p^{rd} \sqrt{\varepsilon}|G|^{r+1}$ of $(r+1)$-tuples $(x_0, \dots, x_r) \in G \times \mathcal{G}^r$ such that $V_{x_0} \cap (V_{x_1} + \dots + V_{x_r}) \not= \{0\}$. By averaging, we get some $r$-tuple $(x_1, \dots, x_r) \in \mathcal{G}^r$ such that for more than $p^{rd} \sqrt{\varepsilon}|G|$ elements $x' \in G$ we have $V_{x'} \cap S \not= \{0\}$, where $S = V_{x_1} + \dots + V_{x_r}$. Since $|V_{x_i}| = p^d$, we see that $|S| \leq p^{rd}$. By averaging over elements of $S \setminus \{0\}$, there exists a non-zero $s \in S$ such that $s \in V_{x'}$ for more than $\sqrt{\varepsilon}|G|$ elements $x' \in G$. However, then $s \in V_{x'} \cap V_{x''}$ for more than $\varepsilon|G|^2$ pairs $(x', x'')$, which is a contradiction.\\
\indent We shall now exploit this fact to get more control over sums of subspaces indexed by configurations of elements of $G$. Write $\varepsilon' = p^{4d} \sqrt{\varepsilon} + 4\varepsilon$, corresponding to the factor in the bound in the observation above for $r\leq 4$.

\begin{claim}\label{achoiceclaim}For all but at most $8\varepsilon' |G|^4$ of quadruples $(a,x,y,z) \in G^4$ we have:
\begin{itemize}
\item[\textbf{(i)}] $|V_a| = |V_x| = |V_y| = |V_{x + y - a}| = |V_z| = p^d$,
\item[\textbf{(ii)}] $|V_a + V_x + V_y| = p^{3d}$, and
\item[\textbf{(iii)}] $V_{x + y - a} \cap (V_x + V_y + V_z) = \{0\}$.
\end{itemize}\end{claim}

\begin{proof}Note that $|V_a + V_x + V_y| = p^{3d}$ if $V_a \cap (V_x + V_y) = \{0\}$, $V_x \cap V_y = \{0\}$ and $|V_a| = |V_x| = |V_y| = p^d$. By assumptions of the proposition and the observation above, all these conditions, as well as $|V_{x + y - a}| = |V_z| = p^d$ and property \textbf{(iii)}, are satisfied for all but at most $8\varepsilon'|G|^3$ triples $(a,x,y) \in G^3$.\end{proof}

Take any $a \in G$ such that the set $\mathcal{T}$ of triples $(x,y,z)$ such that the quadruple $(a,x,y,z)$ satisfies properties in the conclusion of the claim above is of size at least $(1 - 8\varepsilon' )|G|^3$. Let $\mathcal{P}$ be the set of all pairs $(x,y) \in G^2$ such that $|V_x| = |V_y| = |V_{x + y - a}| = p^d$ and the subspace $V_a + V_x + V_y$ has size $p^{3d}$. By Lemma~\ref{basicpropsLSS}, for each $(x,y) \in \mathcal{P}$, we have that 
\begin{equation}V_a + V_x + V_y = V_a + V_x + V_{x + y - a} =  V_a + V_y + V_{x + y - a} = V_x + V_y + V_{x + y - a} = V_a + V_x + V_y + V_{x + y - a}.\label{foursubssumeqs}\end{equation}

Since $|\mathcal{T}| \geq (1 - 8\varepsilon' )|G|^3$, we also have that $|\mathcal{P}| \geq (1 - 8 \varepsilon')|G|^2$.\\

Fix an arbitrary linear isomorphism $\theta \colon \mathbb{F}_p^d \to V_a$. For each $q = (x, y) \in \mathcal{P}$, equalities in~\eqref{foursubssumeqs} imply that the subspaces $V_a$, $V_x$, $V_y$ and $V_{x + y - a}$ satisfy the conditions of Lemma~\ref{addquadsubspaces}. Lemma~\ref{addquadsubspaces} provides us with linear isomorphisms $\theta_q^1 \colon \mathbb{F}_p^d \to V_x$, $\theta_q^2 \colon \mathbb{F}_p^d \to V_y$ and $\theta_q^3 \colon \mathbb{F}_p^d \to V_{x + y - a}$, such that $\theta^1_q + \theta^2_q = \theta + \theta^3_q$. Our next step is to show that these maps do not depend entirely on the pair $q$, but only on the relevant point in $G$.\\
\indent Although $q = (x,y) \in \mathcal{P}$ is a pair, we misuse the notation and write $q_3 = x + y - a$ for the third element associated to $q$.

\begin{claim} For each $i \in [3]$ there exists a set $X_i \subseteq G$ of size $|X_i| \geq (1 - 10\sqrt{\varepsilon'})|G|$ and linear isomorphisms $\phi^i_x \colon \mathbb{F}_p^d \to V_x$ for $x \in X_i$ such that the following holds. For all but at most $30\sqrt{\varepsilon'}|G|^2$ pairs $q \in \mathcal{P}$ we have $q_i \in X_i$ for each $i \in [3]$ and $\theta_q^i = \phi^i_{q_i}$.\end{claim}

\begin{proof}Let us focus first on the case $i = 1$; we shall comment on the other cases afterwards. Let us count the number $N$ of pairs $(q, q') \in \mathcal{P}^2$ such that $q_1 = q'_1$ and $\theta^1_q = \theta^1_{q'}$. Notice that whenever $q = (x, y), q' = (x, y') \in \mathcal{P}$ are two pairs with the same first element, then 
\[\theta^1_q + \theta^2_q - \theta^3_q = \theta = \theta^1_{q'} + \theta^2_{q'} - \theta^3_{q'},\]
so $\theta^2_q - \theta^3_q - \theta^2_{q'} - \theta^3_{q'} + (\theta^1_q - \theta^1_{q'})  = 0$. If the indexing elements satisfy $V_x \cap (V_y + V_{x + y - a} + V_{y'} + V_{x + y' - a}) = \{0\}$, then by Lemma~\ref{indepsubspace} we have $\theta^1_q = \theta^1_{q'}$. On the other hand, Lemma~\ref{basicpropsLSS} shows that $V_y + V_{x + y - a} + V_{y'} + V_{x + y' - a} = V_y + V_{x + y - a} + V_{y'}$. By the choice of $a$, $V_x \cap (V_y + V_{x + y - a} + V_{y'}) = \{0\}$ holds for all but at most $8\varepsilon' |G|^3$ triples $(x,y, y') \in G^3$. From this and Cauchy-Schwarz inequality, we conclude that
\[N \geq |G|^{-1} |\mathcal{P}|^2 - 8\varepsilon' |G|^3 \geq (1 - 24 \varepsilon') |G|^3.\]
Let $X_1$ be the set of all $x \in G$ such that the number of pairs $(q, q')$ above with $x$ as the first element is at least $(1 - 10\sqrt{\varepsilon'}) |G|$. Then $|X_1| \geq (1 - 10\sqrt{\varepsilon'}) |G|$ and we may take $\phi^1_x$ to be the most common map appearing as $\theta^1_q$ among triples with $q_1 = x$.\\

The case $i = 2$ follows by symmetry. For the case $i = 3$, consider $z \in G$ and $q = (x, z + a - x), q' = (x', z + a - x') \in \mathcal{P}$ agreeing in the third associated element, that is $q_3 = q'_3 = z$. Observe that we need $V_z \cap (V_x + V_{z + a - x} + V_{x'} + V_{z + a - x'}) = \{0\}$ property to get the equality of the maps coming from different pairs. By Lemma~\ref{basicpropsLSS}, $V_x + V_{z + a - x} + V_{x'} + V_{z + a - x'} = V_x + V_{z + a - x} + V_{x'}$, so we just need $V_z \cap (V_x + V_{z + a - x} + V_{x'}) = \{0\}$. After a change of variables $y = z + a - x$, this condition has the form guaranteed by Claim~\ref{achoiceclaim}, so the same proof works in this case as well.\end{proof}

Let $\tilde{\mathcal{P}}$ be the set of all pairs $q \in \mathcal{P}$ for which the conclusion of the claim above holds. Let $\on{Hom}(\mathbb{F}_p^d, H)$ be the vector space of all linear homomorphisms from $\mathbb{F}_p^d$ to $H$. Write $\phi^3 \colon X_3 \to \on{Hom}(\mathbb{F}_p^d, H)$ for the map given by $\phi^3(x) = \phi^3_x$. Let us observe that $\phi^3$ respects most additive quadruples. Consider the set $\mathcal{Q}$ of all quadruples $(x,x',y,y') \in G^4$ such that
\begin{itemize}
\item $x, x' \in X_1, y, y' \in X_2, x + y - a, x' + y - a, x + y' - a, x' + y' - a \in X_3$,
\item $(x, y), (x', y), (x, y'), (x', y') \in \tilde{\mathcal{P}}$. 
\end{itemize} 

By the work above, we have that $|\mathcal{Q}| \geq (1 - 300\sqrt{\varepsilon'})|G|^4$. Each such quadruple gives rise to an additive quadruple respected by $\phi^3$. Namely, 
\begin{align*}&\phi^3(x + y - a) + \phi^3(x' + y' - a) - \phi^3(x' + y - a) - \phi^3(x + y' - a)\\
&\hspace{2cm}=\Big(\phi^1_x+ \phi^2_y - \theta\Big) + \Big(\phi^1_{x'}+ \phi^2_{y'} - \theta\Big) - \Big(\phi^1_{x'}+ \phi^2_y - \theta\Big)- \Big(\phi^1_x+ \phi^2_{y'} - \theta\Big) = 0.\end{align*}
On the other hand, each additive quadruple in $X_3$ can arise at most $|G|$ times in this way. Hence, $\phi^3$ respects at least $(1 - 300\sqrt{\varepsilon'})|G|^3$ additive quadruples in $X_3$.\\
\indent We are now in position to apply Lemma~\ref{elemnearhomm}. It provides us with an affine map $\Phi \colon G \to \on{Hom}(\mathbb{F}_p^d, H)$ such that $\Phi(x) = \phi^3(x)$, and thus $\on{Im} \Phi(x) = \on{Im} \phi^3_x = V_x$, holds for at least $(1 - \varepsilon'') |G|$ of $x \in X_3$, where $\varepsilon'' \leq 100\sqrt[8]{\varepsilon'}$, provided $\varepsilon' \leq 10^{-15}$, proving the proposition.\end{proof}

To prove Theorem~\ref{bilinearSetsStructure} it thus remains to locate a sufficiently quasirandom piece. The following regularity lemma for transverse sets achieves this. The proof is based on dependent random choice and the algebraic structure is essential for ensuring the good bounds.

\begin{lemma}\label{reglemma}Let $A \subset G \times H$ be a transverse set of size $\delta |G||H|$ and let $\varepsilon > 0$ be a parameter. Then there exist a non-negative integer $d \leq \lceil \log_p (10^4 \delta^{-4}) \rceil$ and subspaces $U \leq G$ and $V \leq H$ of codimensions $\on{codim}_G U \leq O\Big((\log(2\delta^{-1})^2 \log(2\varepsilon^{-1})\Big)$ and $\on{codim}_H V \leq O(\log(2 \delta^{-1})^2)$ such that:
\begin{itemize}
\item[\textbf{(i)}] for at least $(1 - \varepsilon)|U|$ of elements $x \in U$ we have $|A_{x \bcdot} \cap V| = p^{-d} |V|$, and
\item[\textbf{(ii)}] for at least $(1 - 3\varepsilon)|U|^2$ of pairs $(x_1, x_2) \in U^2$ we have $|A_{x_1 \bcdot} \cap A_{x_2 \bcdot} \cap V| = p^{-2d} |V|$.
\end{itemize}
\end{lemma}

\begin{proof}Let $Y$ be the set of all $y \in H$ such that $|A_{\bcdot y}| \geq \frac{\delta}{2} |G|$. By averaging, we have $|Y| \geq \frac{\delta}{2}|H|$. Pick $y \in Y$ uniformly at random. Then $\mathbb{P}(x \in A_{\bcdot y}) = \frac{|A_{x \bcdot} \cap Y|}{|Y|}$. Let $X_{\text{bad}}$ be the set of elements $x \in A_{\bcdot y}$ such that $|A_{x \bcdot}| \leq \frac{\delta^2}{100}|H|$. Then
\[\exx \Big(|A_{\bcdot y}| - 10 |X_{\text{bad}}|\Big) \geq \frac{\delta}{2}|G| - 10\sum_{x \in G} \mathbb{P}(x \in X_{\text{bad}}) = \frac{\delta}{2}|G| - 10\sum_{x \in G}\id(|A_{x \bcdot}| \leq \frac{\delta^2}{100}|H|) \mathbb{P}(x \in A_{\bcdot y}) \geq \frac{\delta}{4}|G|.\]

Hence, there is a choice of $y_0$ such that $|A_{\bcdot y_0}| \geq \frac{\delta}{2}|G|$ and $|A_{x \bcdot}| \geq \frac{\delta^2}{100}|H|$ holds for at least $\frac{9}{10}|A_{\bcdot y_0}|$ elements $x \in A_{\bcdot y_0}$. Hence, for an arbitrary $x \in A_{\bcdot y_0}$ we may find $x' \in A_{\bcdot y_0}$ such that $|A_{x' \bcdot}| \geq \frac{\delta^2}{100}|H|$ and $|A_{x + x' \bcdot}| \geq \frac{\delta^2}{100}|H|$ both hold. As $A$ is transverse we have $\frac{\delta^4}{10^4} |H| \leq |A_{x' \bcdot} \cap A_{x + x' \bcdot}| \leq |A_{x \bcdot}|$. Thus, $|A_{x \bcdot}| \geq \frac{\delta^4}{10^4} |H|$ holds for all elements $x \in A_{\bcdot y_0}$.\\ 
\indent Similarly, there exists an element $x_0 \in G$ such that  $|A_{x_0 \bcdot}| \geq \frac{\delta}{2}|H|$ and $|A_{\bcdot y}| \geq \frac{\delta^4}{10^4} |G|$ holds for all elements $y \in A_{x_0 \bcdot}$. Set $U_0 = A_{\bcdot y_0}$, $V_0 = A_{x_0 \bcdot}$ and $d_0 = \lceil \log_p (10^4 \delta^{-4}) \rceil$. Hence, $U_0$ and $V_0$ have codimension at most $d_0$ in $G$ and $H$ respectively, for each $x \in U_0$ we have $\on{codim}_{V_0} (A_{x \bcdot} \cap V_0) \leq d_0$ and for each $y \in V_0$ we have $\on{codim}_{U_0} (A_{\bcdot y} \cap U_0) \leq d_0$.\\

We now perform an iterative procedure in which at step $i$ we find further subspaces $U_i \leq U_0$, $V_i \leq V_0$, of codimension $\on{codim}_{U_0} U_i \leq O(i d_0 \log(2 \varepsilon^{-1}))$ and $\on{codim}_{V_0} V_i \leq O(i d_0)$, such that the proportion of elements $x \in U_i$ such that
\begin{equation}\on{codim}_{V_i} (A_{x \bcdot} \cap V_i) \leq d_0 - i\label{fewsmallsubspaces}\end{equation}
is at least $1 - \frac{1}{2} \varepsilon$. The elements $x$ in~\eqref{fewsmallsubspaces} are those that have $A_{x \bcdot} \cap V_i$ at least as large as we would like, thus we guarantee that most subspaces restricted to $V_i$ have at least the desired size, if not equal to the one required by property \textbf{(i)}. The procedure terminates when we find subspaces $U_i$ and $V_i$ satisfying the claimed properties in the lemma, with $d$ taken to be $d_0 - i$. Suppose that we have subspaces $U_i, V_i$ for some $i \geq 0$ (when $i=0$ we know that $U_0$ and $V_0$ have the desired properties). If both properties \textbf{(i)} and \textbf{(ii)} hold for $U_i$ and $V_i$ and integer $d_0-i$, we are done. We now distinguish between two cases, depending on which one of the two properties fails.\\

\noindent \textbf{Property \textbf{(i)} fails.} We thus have at least $\varepsilon|U_i|$ elements $x \in U_i$ such that $\on{codim}_{V_i} (A_{x \bcdot} \cap V_i) \not= d_0-i$. Since $\on{codim}_{V_i} (A_{x \bcdot} \cap V_i) \leq d_0 - i$ holds for at least $(1 - \varepsilon/2)|U_i|$ elements $x \in U_i$, we conclude that $\on{codim}_{V_i} (A_{x \bcdot} \cap V_i) \leq d_0 - i - 1$ holds for at least $\frac{\varepsilon}{2}|U_i|$ elements $x \in U_i$. Let $r = \lceil2\log_p(4 \varepsilon^{-1}) \rceil$ and take $y_1, \dots, y_{r}$ uniformly and independently at random from $V_i$. Let $U_{i+1} = U_i \cap A_{\bcdot y_1} \cap \dots\cap A_{\bcdot y_{r}}$ and set $V_{i+1} = V_i$. Let $X_{\text{bad}}$ be the set of all elements $x \in U_{i + 1}$ such that $\on{codim}_{V_i} (V_i \cap A_{x \bcdot}) \geq d_0-i$. Observe that, for each $x \in U_i$, $\mathbb{P}(x \in U_{i+1}) = \Big(\frac{|A_{x \bcdot} \cap V_i|}{|V_i|}\Big)^r = p^{ - r \on{codim}_{V_i} (V_i \cap A_{x \bcdot})}$. Then
\[\exx\Big(|U_{i+1}| - 2\varepsilon^{-1}|X_{\text{bad}}|\Big) \geq  p^{-r (d_0 - i - 1)} \frac{\varepsilon}{2} |U_i| - 2 \varepsilon^{-1} p^{-r(d_0 - i)} |U_i| = 2 \varepsilon^{-1}p^{-r(d_0 - i)} |U_i| \Big(p^r \frac{\varepsilon^2}{4} - 1\Big)  > 0,\]
by our choice of $r$. Hence, there is a choice of $y_1, \dots, y_{r}$ producing the desired pair of subspaces $U_{i+1}$ and $V_{i+1}$. Note that $\on{codim}_{U_i} U_{i+1} \leq r d_0 = O(d_0 \log_p(2 \varepsilon^{-1}))$ and $\on{codim}_{V_i} V_{i+1} = 0$.\\

\noindent \textbf{Property \textbf{(ii)} fails, but \textbf{(i)} holds.} In this case have at least $3\varepsilon|U_i|^2$ pairs of elements $(x_1, x_2) \in U^2_i$ such that $\on{codim}_{V_i}( A_{x_1 \bcdot} \cap A_{x_2 \bcdot} \cap V_i) \not= 2(d_0-i)$. We call such a pair $(x_1, x_2)$ \emph{irregular}. Let $R$ be the set of all $a \in U_i$ such that $\on{codim}_{V_i} (A_{a \bcdot} \cap V_i) = d_0-i$. Since property \textbf{(i)} holds, we have $|R| \geq (1 - \varepsilon) |U_i|$. The number of irregular pairs $(x_1, x_2)$ such that $x_1 \notin R$ is at most $\varepsilon |U_i|^2$. Hence, there at least $2\varepsilon|U_i|^2$ irregular pairs $(x_1, x_2) \in R \times U_i$, so by averaging, there exists $a \in R$ such that $(a,x)$ is irregular for at least $2\varepsilon|U_i|$ choices of $x \in U_i$.\\
\indent By assumption~\eqref{fewsmallsubspaces}, at least $\varepsilon|U_i|$ of such elements $x \in U_i$ also have the property that $\on{codim}_{V_i} A_{x \bcdot} \cap V_i \leq d_0-i$. Note that 
\begin{align*}\on{codim}_{V_i}( A_{a \bcdot} \cap A_{x \bcdot} \cap V_i) = &\dim V_i - \dim (A_{a \bcdot} \cap A_{x \bcdot} \cap V_i)\\
 = &\dim V_i - \dim (A_{a \bcdot} \cap V_i)  - \dim (A_{x \bcdot} \cap V_i)  + \dim \Big( (A_{a \bcdot} \cap V_i) + (A_{x \bcdot} \cap V_i) \Big)\\
= & \on{codim}_{V_i}A_{a \bcdot} + \on{codim}_{V_i}A_{x \bcdot} - \dim V_i + \dim \Big( (A_{a \bcdot} \cap V_i) + (A_{x \bcdot} \cap V_i) \Big)\\
\leq &\on{codim}_{V_i}A_{a \bcdot} + \on{codim}_{V_i}A_{x \bcdot} \leq 2(d_0 - i),\end{align*}
where we used $(A_{a \bcdot} \cap V_i) + (A_{x \bcdot} \cap V_i) \subseteq V_i$ in the inequality between the last two lines. Since $(a,x)$ is irregular, we have $\on{codim}_{V_i}( A_{a \bcdot} \cap A_{x \bcdot} \cap V_i) < 2(d_0-i)$. Setting $V_{i+1} = V_i \cap A_{a \bcdot}$, for such an $x$ we have \begin{align*}\on{codim}_{V_{i+1}} (A_{x \bcdot} \cap V_{i+1}) =& \dim(V_{i+1}) - \dim(A_{x \bcdot} \cap V_{i+1}) = \dim(V_i \cap A_{a \bcdot}) - \dim( V_i \cap A_{x \bcdot} \cap A_{a \bcdot}) \\
= &\Big(\dim V_i - \on{codim}_{V_i} (A_{a \bcdot} \cap V_i)\Big) - \Big(\dim V_i - \on{codim}_{V_i}( A_{a \bcdot} \cap A_{x \bcdot} \cap V_i)\Big) \\
= &\on{codim}_{V_i}( A_{a \bcdot} \cap A_{x \bcdot} \cap V_i) - \on{codim}_{V_i} (A_{a \bcdot} \cap V_i) \leq d_0 - i - 1.\end{align*}
We may now repeat the same final step of the previous case, with the same choice of $r = \lceil2\log_p(4 \varepsilon^{-1}) \rceil$. Take $y_1, \dots, y_{r}$ uniformly and independently at random from $V_{i+1}$ and let $U_{i+1} = U_i \cap A_{\bcdot y_1} \cap \dots\cap A_{\bcdot y_{r}}$. By the calculations in the previous step, there exists a choice of $y_1, \dots, y_r$ such that $\on{codim}_{V_{i+1}} (A_{x \bcdot} \cap V_{i+1}) \leq d_0 - i - 1$ for at least $\Big(1 - \frac{1}{2} \varepsilon\Big)|U_{i+1}|$ elements $x \in U_{i+1}$. In this case, the codimension bounds are $\on{codim}_{U_i} U_{i+1} \leq r d_0 = O(d_0 \log_p(2 \varepsilon^{-1}))$ and $\on{codim}_{V_i} V_{i+1} \leq d_0$.\end{proof}

We are now in position to prove Theorem~\ref{bilinearSetsStructure}.

\begin{proof}[Proof of Theorem~\ref{bilinearSetsStructure}]Let $A$ be the given transverse set of density $\delta$. Apply Lemma~\ref{reglemma} with parameter $\varepsilon = 2^{-160} p^{-32 \lceil \log_p(10^4 \delta^{-4})\rceil}$ to get a positive integer $d \leq \lceil \log_p(10^4 \delta^{-4})\rceil$ and subspaces $U \leq G$ and $V \leq H$ such that:
\begin{itemize}
\item[\textbf{(i)}] $\on{codim}_G U \leq O(\log^3_p(2 \delta^{-1}))$ and $\on{codim}_H V \leq O(\log^2_p(2 \delta^{-1}))$,
\item[\textbf{(ii)}] for at least $(1 - \varepsilon/3)|U|$ of elements $x \in U$ we have $|A_{x \bcdot} \cap V| = p^{-d} |V|$, and
\item[\textbf{(iii)}] for at least $(1 - \varepsilon/3)|U|^2$ of pairs $(x_1, x_2) \in U^2$ we have $|A_{x_1 \bcdot} \cap A_{x_2 \bcdot} \cap V| = p^{-2d} |V|$.
\end{itemize}

Note that $\varepsilon \leq  2^{-160} p^{-32 d} < 10^{-32} p^{-8d}$. Fix an inner product $\cdot$ on $V$ and, for each $x \in U$, define subspace $V_x = (A_{x \bcdot} \cap V)^\perp$, where ${}^\perp$ indicates orthogonal complement in $V$ with respect to $\cdot$. As $A$ is a transverse set, the collection $(V_x)_{x \in U}$ is a linear system of subspaces. By properties \textbf{(ii)} and \textbf{(iii)} we have that $|V_x| = p^d$ holds for at least $(1 - \varepsilon/3)|U|$ of elements $x \in U$, and three equalities $|A_{x_1 \bcdot} \cap V| = p^{-d} |V|$, $|A_{x_2 \bcdot} \cap V| = p^{-d} |V|$ and $|A_{x_1 \bcdot} \cap A_{x_2 \bcdot} \cap V| = p^{-2d} |V|$ hold simultaneously for at least $(1 - \varepsilon)|U|^2$ of pairs $(x_1, x_2) \in U^2$. These three equalities together imply that 
\[|V_{x_1} \cap V_{x_2}| = \frac{|V_{x_1}| |V_{x_2}|}{|V_{x_1} + V_{x_2}|} = \frac{\frac{|V|}{|A_{x_1 \bcdot} \cap V|} \frac{|V|}{|A_{x_2 \bcdot} \cap V|}}{\frac{|V|}{|(V_{x_1} + V_{x_2})^\perp|}} = \frac{|V| |(V_{x_1} + V_{x_2})^\perp|}{p^{-2d} |V|^2} = \frac{|V_{x_1}^\perp \cap V_{x_2}^\perp|}{p^{-2d} |V|} = \frac{|A_{x_1 \bcdot}\cap A_{x_2 \bcdot} \cap V|}{p^{-2d} |V|} = 1,\]
hence $V_{x_1} \cap V_{x_2} = \{0\}$.\\

We may apply Proposition~\ref{propfindingbilstruct} to linear system of subspaces $(V_x)_{x \in U}$ inside $V$ to find a bilinear map $\Phi \colon U \times \mathbb{F}_p^d \to V$ and a linear map $\Psi \colon \mathbb{F}_p^d \to V$ such that $\{\Phi(x, \lambda) + \Psi(\lambda) \colon \lambda \in \mathbb{F}_p^d\} = V_x$ holds for at least $(1 - 200p^d \sqrt[16]{\varepsilon})|G|$ elements $x \in U$. By our choice of $\varepsilon$, we have $200p^d \sqrt[16]{\varepsilon} \leq \frac{1}{3} p^{-d}$. Hence, there exists a set $X \subseteq U$, of size $|X| \geq (1 - p^{-d}/3) |U|$, such that $\{\Phi(x, \lambda) + \Psi(\lambda) \colon \lambda \in \mathbb{F}_p^d\} = (A_{x \bcdot} \cap V)^\perp$ holds for all $x \in X$.\\
\indent Taking orthogonal complements in $V$, we get $\{\Phi(x, \lambda) + \Psi(\lambda) \colon \lambda \in \mathbb{F}_p^d\}^\perp = A_{x \bcdot} \cap V$. Let $V' = V \cap (\on{Im} \Psi)^\perp$, and let $\beta_i(x,y) = \Phi(x, e_i) \cdot y$ for $i \in [d]$, where $e_1, \dots, e_d$ is the standard basis of $\mathbb{F}_p^d$. Then for each $x \in X$, we have $\{y \in V' \colon \beta(x,y) = 0\} \subseteq A_{x \bcdot}$. We claim that in fact
\[\{(x,y) \in U \times V' \colon \beta(x,y) = 0\} \subseteq A.\]
To see this, take any $(x,y) \in U \times V'$ such that $\beta(x,y) = 0$. The subspace $\{u \in U \colon \beta(u,y) = 0\}$ has codimension at most $d$ inside $U$. As $|X| \geq (1 - p^{-d}/3) |U|$, we conclude that $|X \cap \{u \in U \colon \beta(u,y) = 0\}| \geq \frac{2}{3} |\{u \in U \colon \beta(u,y) = 0\}|$. Thus, as $\beta(x,y) = 0$, we have that $(X \cap \{u \in U \colon \beta(u,y) = 0\}) + x \subseteq \{u \in U \colon \beta(u,y) = 0\}$ as well, so $(X \cap \{u \in U \colon \beta(u,y) = 0\}) + x$ and $(X \cap \{u \in U \colon \beta(u,y) = 0\})$ intersect, at some element $u$, say. Thus, $u, u -x \in X$ and $\beta(u, y) = \beta(u - x, y) = 0$. By properties of elements in $X$, we have that $(u, y), (u - x, y) \in A$. But $A$ is a transverse set, so $(x,y) \in A$, as desired.\\
\indent Finally, let us remark that terms of the form $\log_p(2\delta^{-1})$ in the bounds above can be replaced by more esthetically pleasing $\log_p(\delta^{-1})$. This stems from the fact that if the density $\delta$ is at least $\frac{15}{16}$, then our transverse set $A$ is in fact the full product $G \times H$. We leave this simple observation as an exercise to the reader.\end{proof}

\end{document}